\documentclass[12pt]{article}
\usepackage{amssymb}
\usepackage{amsfonts}
\usepackage{amsmath}
\usepackage{amsmath}
\usepackage{amssymb}
\usepackage{verbatim}
\usepackage{enumerate}
\usepackage{color}
\def\mE{{\mathbb E}}
\usepackage[active]{srcltx}
\usepackage{mathrsfs}
\usepackage{stmaryrd}
\include{srctex}
\allowdisplaybreaks[4]
\newtheorem{theorem}{Theorem}[section]

\def\mN{{\mathbb N}}

\newtheorem{Examples}[theorem]{Example}

\newtheorem{lemma}[theorem]{Lemma}

\newtheorem{problem}[theorem]{Problem}

\newtheorem{remark}[theorem]{Remark}

\let\Section=\section
\def\section{\setcounter{equation}{0}\Section}
\newenvironment{proof}[1][Proof]{\textbf{#1.} }{\ \rule{0.5em}{0.5em}}

\def\sL{{\mathscr L}}

\def\e{{\mathrm{e}}}

\def\RR{\mathbb{R}}
\def\NN{\mathbb{N}}
\def\EE{\mathbb{E}}

\def\dif{{\mathord{{\rm d}}}}

\def\<{{\langle}}
\def\>{{\rangle}}

\def\et{{\eta}}

\def\mE{{\mathbb E}}

\def\mI{{\mathbb I}}

\def\mN{{\mathbb N}}

\def\mR{{\mathbb R}}

\def\bt{\begin{theorem}}
\def\et{\end{theorem}}
\def\bl{\begin{lemma}}
\def\el{\end{lemma}}
\def\br{\begin{remark}}
\def\er{\end{remark}}
\def\geq{\geqslant}
\def\leq{\leqslant}
\def\bx{\begin{Examples}}
\def\ex{\end{Examples}}

\begin{document}

\title{Pathwise uniqueness for a class of SPDEs driven by cylindrical $\alpha$-stable processes
\thanks{ Supported by the NNSF of China (No.: 11601196, 11271169) and the Priority Academic Program Development of Jiangsu Higher Education Institutions.\newline
Keywords and phrases: Pathwise uniqueness; Stochastic partial differential equation; $\alpha$-stable process }}
\author{{Xiaobin Sun}$$\footnote{E-mail: xbsun@jsnu.edu.cn}\;
\  \ \ {Longjie Xie}$$\footnote{E-mail: xlj.98@whu.edu.cn}\;
\  \ \ {Yingchao Xie}$$\footnote{E-mail: ycxie@jsnu.edu.cn}\;
\\
 \small   School of Mathematics and Statistics, Jiangsu Normal University, Xuzhou 221116, China.}\,

\date{}
\maketitle

\begin{abstract}
We show the pathwise uniqueness for stochastic partial differential equation driven by a cylindrical $\alpha$-stable process with H\"older continuous drift, thus obtaining an infinite dimensional generalization of the result of Priola [Osaka J. Math., 2012] in the case $H=\mR^d$. The proof is based on an infinite dimensional Kolmogorov equation with non-local operator.
\end{abstract}

\section{Introduction}

In this paper, we consider the following stochastic partial differential equation (SPDE) in Hilbert space:
\begin{equation}\label{main equation}\left\{\begin{array}{l}
\displaystyle dX_t=AX_tdt+B(X_t)dt+dZ_t,\\
X_0=x\in H.\end{array}\right.
\end{equation}
The objects are: a separable Hilbert space $H$ with the inner product $\langle\cdot,\cdot\rangle$
and norm $|\cdot|$, a self-adjoint operator $A:\mathcal{D}(A)\subset H\to H$ which is the infinitesimal generator of a linear strongly continuous
semigroup $(e^{tA})_{t\ge0}$. The process $Z=(Z_t)_{t\geq0}$ is a cylindrical $\alpha$-stable process with $\alpha\in(1,2)$ defined on some probability space $(\Omega,\mathcal{F},\mathbb{P})$ equipped with the filtration $\{\mathcal{F}_t,t\geq0\}$. We shall only assume $B: H\rightarrow H$ is H\"{o}lder continuous with certain power.
Our aim is to prove the pathwise uniqueness for SPDE (\ref{main equation}), which is an infinite dimensional generalization of the result by Priola \cite{P} in the case $H=\mR^d$.

\vspace{2mm}

Currently, there is an increasing interests in understanding the regularization effects of noise to the deterministic equations. We refer to \cite{Fl} for a review on this direction. When $H=\mR^d$ and $A=0$, SPDE (\ref{main equation}) is just the stochastic differential equation (SDE):
\begin{align}
\dif X_t=b(X_t)\dif t+\dif Z_t, \quad X_0=x\in\mR^d.  \label{sde}
\end{align}
In the case that $Z_t$ is a Brownian motion, Veretennikov \cite{Ve} first proved that SDE (\ref{sde}) has a unique global strong solution $X_t(x)$ if $b$ is bounded and measurable. Later, Krylov and R\"ockner \cite{Kr-Ro}
shows that SDE (\ref{sde}) has a unique strong solution when $b\in L^p(\mR^d)$ with $p>d$.
The case that $Z_t$ is a pure jump symmetric $\alpha$-stable process has more difficulties. When $\alpha\geq 1$ and
\begin{align*}
b\in C_b^\beta(\mR^d)\quad \text{with}\quad\beta>1-\frac{\alpha}{2},
\end{align*}
it was proved by Priola \cite{P} that there exists a unique strong solution $X_t(x)$ to SDE (\ref{sde}) for each $x\in\mR^d$. Recently, Zhang \cite{Zh00} obtained the pathwise uniqueness to SDE (\ref{sde}) when $\alpha>1$, $b$ is bounded and belongs to certain fractional Sobolev space.
See also \cite{Fe-Fl-2,Fe-Fl-3,M-M-N-P-Z,Pri2,Pri3,Zh3} and references therein for related results concerning SDEs with irregular coefficients.

\vspace{2mm}
For the infinite dimensional case, when $Z_t$ is a cylindrical Winer noise and $B$ is H\"older continuous, the authors in \cite{DF} showed that there exists a unique strong solution to SPDE (\ref{main equation}) for every $x\in H$. The extension of Veretennikov's result to infinite dimensional with $B$ bounded was done in \cite{D-F-P-R}. However, the uniqueness holds only for almost all starting point $x\in H$. See also \cite{D-F-P-R-2,D-F-R-V,W,Wa} and references therein.

\vspace{2mm}
Usually, SPDEs with jumps have more wide range of applications, we refer to the recent monograph \cite{PZ}. When the drift $B$ in (\ref{main equation}) is Lipschitz continuous, the existence and uniqueness of solution can be easily obtained by a fixed point argument in \cite{EJ1}. Later on, more concreted SPDEs driven by cylindrical $\alpha$-stable processes have been studied, see \cite{DXZ, DXZ1, EJ2, EJ3, EJ1, SX, XLH}.

\vspace{2mm}
As far as we know, there is still no work on the pathwise uniqueness for SPDEs driven by pure jump L\'evy process with irregular coefficient. The main difficult is that, from the analytic point of view, the generator of pure jump L\'evy process is a non-local operator; and from the probability point of view, processes with jumps are more complicated than the continuous diffusion processes.  We shall study the pathwise uniqueness of SPDEs (\ref{main equation}) with H\"older continuous drift by solving the infinite dimensional Kolmogorov equation with non-local operator.

\vspace{3mm}
The paper proceeds as follows: In section 2, we state the main result. In Section 3, we study the regularity of the Ornstein-Uhlenbeck semigroup and solve the corresponding Kolmogorov equation. Finally, the proof of main result is given in Section 4.

 \vspace{3mm}
Throughout our paper, we use the following convention: $C$ with or without subscripts will denote a positive constant, whose value may change in different places, and whose dependence on parameters can be traced from calculations.

\section{Preliminaries and main result}

Given $\beta\in (0,1]$, we denote by $C^{\beta}_b(H,H)$ the usual H\"older space of functions $G(x): H\rightarrow H$ with norm
$$
\|G\|_{\beta}:=\sup_{x\in H}|G(x)|+\sup_{x\neq y\in H}\frac{|G(x)-G(y)|}{|x-y|^{\beta}},
$$
and let $\|G\|_{0}=\sup_{x\in H}|G(x)|$. Similar, for given $\beta\in (1,2]$, the space $C^{\beta}_b(H,H)$ denotes functions satisfying
$$
\|G\|_{\beta}:=\|G\|_1+\sup_{x\neq y\in H}\frac{\|DG(x)-DG(y)\|}{|x-y|^{\beta-1}}<\infty,
$$
where $\|\cdot\|$ is the operator norm.

\vspace{2mm}
The cylindrical $\alpha$-stable process $Z$ is denoted by
$$
Z_t=\sum_{n\geq1}\beta_{n}Z^{n}_{t}e_n,\quad t\geq 0,
$$
where $\{\beta_n\}_{n\ge1}$ is a given sequence of positive numbers, $\{e_n\}_{n\geq1}$ is a complete orthonormal basis of $H$,
and $\{Z^{n}_t\}_{n\ge1}$ are independent one dimensional rotationally symmetric $\alpha$-stable process, i.e., the L\'evy measure of $Z_t^n$ is given by
\begin{align*}
\nu(\dif z)=\frac{c_{\alpha}}{|z|^{1+\alpha}}\dif z,
\end{align*}
where $c_{\alpha}$ is a positive constant. By L\'evy-It\^o's decomposition, one has
$$Z^{n}_{t}=\int_{|x|\leq1}x\widetilde{N}^{(n)}(t,dx)+\int_{|x|>1}x N^{(n)}(t,dx),$$
where $N^{(n)}(t,\Gamma)$ is the Possion random measure, i.e.,
$$
N^{(n)}(t,\Gamma)=\sum_{s\leq t}I_{\Gamma}(Z^{n}_s-Z^{n}_{s-}), \quad\forall t>0, \Gamma\in\mathcal{B}(\mathbb{R}\setminus\{0\})
$$
and $\widetilde{N}^{(n)}(t,\Gamma)$ is compensated Poisson measure, i.e.,
$$
\widetilde{N}^{(n)}(t,\Gamma)=N^{(n)}(t,\Gamma)-t\nu(\Gamma).
$$

\vspace{2mm}
Consider equation (\ref{main equation}), we make the following assumptions:
\begin{enumerate}[{(i)}]

\item $\{e_n\}_{n\geq1}\subset\mathcal{D}(A)$, $Ae_n=-\gamma_{n}e_n$ with $\gamma_n>0$ and $\gamma_n\uparrow \infty.$

\item $\sum_{n\geq1}\beta^{\alpha}_{n}<\infty.$

\item $\sum_{n\geq1}\frac{1}{\gamma_n}<\infty.$

%\end{enumerate}

%Blabla......why we need (i), why we need (ii), why we need (iii)....

%We also need the following condition when studying the smoothing property of the Ornstein-Uhlenbeck semigroup.
%\begin{enumerate}[{(iv)}]

\item There exists a $\gamma\in(1, \alpha]$ such that for any $r<\gamma$ and $\lambda>0$, we have
\begin{align}
\int_0^\infty\e^{-\lambda t}\Lambda_t^r \dif t<\infty,    \label{im9}
\end{align}
where
\begin{align}
\Lambda_t:=\sup_{n\geq1}\frac{e^{-{\gamma_n t}}\gamma^{1/\alpha}_n}{\beta_n}<\infty. \label{im10}
\end{align}
\end{enumerate}

The main result of our paper is stated as follows:
\bt\label{main1}
Assume that (i)-(iv) hold and $B\in C^{\beta}_{b}(H,H)$ for some $\beta\in(1+\alpha/2-\gamma,1)$.
Then, SPDE (\ref{main equation}) has a unique strong solution for each $x\in H$.
\et

\begin{remark}
The condition (ii) is the sufficient and necessary condition for $Z_t$ is a L\'evy process in $H$ (see \cite{LZ}). (\ref{im9}) and  (\ref{im10}) in condition (iv) are used to study the smoothing property of the Ornstein-Uhlenbeck semigroup and regularity of the solution of the corresponding Kolmogorov equation in section 3.
\end{remark}

We give some examples to illustrate our result.
\bx
In the case that $H=\mR^d$, we can take $\gamma$ in (\ref{im9}) equals to $\alpha$, and this means that we need  $B\in C^{\beta}_{b}(\mR^d)$ for $\beta\in(1-\alpha/2,1)$. Thus, we go back to \cite{P}.
\ex

\vskip 0.3cm
\bx
We set $H=L^2(D)$,  where $D=[0,\pi]$, and denote by $\partial D$ the boundary of $D$.
Considering the stochastic Reaction-Diffusion Equation on $D$.
\begin{equation}\left\{\begin{array}{l}\label{example of theorem}
\displaystyle
dX(t,\xi)=\Delta^{p}_{\xi}X(t,\xi)dt+b(X(t,\xi))dt+dZ_t,\\
X(t,\xi)=0, \quad t\geq 0,\quad \xi\in\partial D\\
X(0,\xi)=x(\xi),\quad \xi\in D, x\in H,\end{array}\right.
\end{equation}
where $Z$ is a cylindrical $\alpha$-stable process with $\alpha\in (1, 2)$, $\Delta^p_{\xi}$ is pseudodifferential operator with $p\geq 1$ and $\frac{1}{2p}+1<\alpha$, $b:\RR\rightarrow \RR$ is a bounded and H\"{o}lder continuous with index $\beta$, where $\beta$ will be determined later.
Put
$$Ax=\Delta^{p}_{\xi}x,\quad x\in \mathcal{D}(A)=H^{2p}(D) \cap H^1_0(D),$$
where $H^{2p}(D)$ is the usual Sobolev space, and
$$B(x)=b(x(\cdot)),\quad x\in H.$$
Operator $A$ possesses a complete orthonormal system of eigenfunctions namely
$$e_n(\xi)=(\sqrt{2/\pi})\sin(n\xi),\quad\xi\in[0,\pi],$$
where $n\in \mathbb{N}$. The corresponding eigenvalues are $-\gamma_n$, where $\gamma_n=n^{2p}$.

Moreover, notice that $\frac{1}{2p}+1<\alpha$, if choosing $\beta_n=C\gamma^{-r}_n$ with $r\in\left(\frac{1}{2p\alpha}, \frac{\alpha-1}{\alpha}\right)$, then it is easy to verify conditions (i)-(iii) hold. Also by Remark \ref{example} below, $\Lambda_t=\sup_{n\geq1}\frac{e^{-{\gamma_n t}}\gamma^{1/\alpha}_n}{\beta_n}\leq\frac{C}{t^{r+\frac{1}{\alpha}}}$, and taking $\gamma=\frac{\alpha}{\alpha r+1}\in(1,\alpha)$
\begin{equation}
\int_0^\infty\e^{-\lambda t}\Lambda_t^q \dif t<\infty,  \quad \forall q<\gamma,\nonumber
\end{equation}
which verifies condition (iv).
Consequently, taking H\"{o}lder index $\beta\in \left(1-\frac{\alpha(1-\alpha r)}{2(\alpha r+1)},1\right)$, then (\ref{example of theorem}) has a unique strong solution by Theorem \ref{main1}.

For instance, if $p=1$, i.e., $\Delta^p_{\xi}$ is the Laplace operator, then for any $\alpha \in (\frac{3}{2},2)$, taking
$$\beta_n=C\gamma^{-r}_n \quad \text{with} \quad r\in\left(\frac{1}{2\alpha}, \frac{\alpha-1}{\alpha}\right)$$
and
$$
\gamma=\frac{\alpha}{\alpha r+1}\in(1,\alpha),\quad \beta\in \left(1-\frac{\alpha(1-\alpha r)}{2(\alpha r+1)},1\right).
$$
\ex

\section{Ornstein-Uhlenbeck semigroup and corresponding Kolmogorov equation}

\subsection{$H$-valued Ornstein-Uhlenbeck semigroup}

We first consider the following linear equation:
\begin{eqnarray}
dY_t=AY_tdt+dZ_t,\quad Y_0=x\in H.\label{OUEq}
\end{eqnarray}
Throughout this subsection, we assume that (i), (\ref{im10}) hold and

\begin{enumerate}[{(ii)'}]

\item There exists a positive constant $C_\alpha$ such that
$$
\sum_{n=1}^\infty\frac{\beta_n^\alpha}{\gamma_n}\leq C_\alpha<\infty.
$$

\end{enumerate}

Notice that condition (ii') is weaker than (ii). Under the assumptions (i) and (ii)', it is well-known that equation (\ref{OUEq}) has a unique mild solution for any initial value $x\in H$, which is given by
$$Y^x_t=e^{tA}x+Z_{A}(t),$$
where $Z_{A}(t)=\int^{t}_{0}e^{(t-s)A}dZ_s$. The solution $Y_t$ is called the Ornstein-Uhlenbeck process and has received a lot of attentions. Let $R_t: B_b(H)\rightarrow B_b(H)$ be the corresponding semigroup defined by
$$
R_t f(x)=\EE[f(Y^{x}_t)]=\int_{H}f(y)\mu^x_t(dy), \quad x\in H,\quad f\in B_b(H),\quad t\geq 0,
$$
where $\mu^x_t$ is the law of $Y^x_t$ and $B_b(H)$ consists of all bounded functions $f: H\rightarrow\mR$. This semigroup has been also studied under the name of generalized Mehler semigroup.
It was show in \cite{EJ1} that $\mu^x_t$ can be seen as Borel product measures in $\RR^{\infty}$, i.e.,
$$
\mu^x_t=\prod_{k\geq 1}\mu^{x_k}_t,
$$
where $x_k=\langle x, e_k\rangle$, and $\mu^{x_k}_t$ is a probability measure on $\mR$ with density function
$$
\frac{1}{c_{k}(t)}p_{\alpha}\left(\frac{z_k-e^{-\gamma_k t}x_k}{c_k(t)}\right), \quad\text{and}\quad c_k(t):=\beta_k\left(\frac{1-e^{-\alpha\gamma_k t}}{\alpha \gamma_k}\right)^{1/\alpha}
$$
here, $p_{\alpha}$ is the density of random variable $Z^n_1$.

\vskip 0.3cm
The next result shows that $R_t$ has a smoothing effect and the estimates of first and second derivative are given, which is a important step to prove our main result. Below, we also use $\langle\cdot,\cdot\rangle$ as the action of two elements without confuse.

\begin{theorem}\label{T3.1}
For every $x, h, g\in H$ with $|h|\leq 1$, $|g|\leq 1$, set $h_k=\langle h, e_k\rangle$, $g_k=\langle g, e_k\rangle$. Then, for any $t>0$ and $f\in B_b(H)$, we have $R_t f\in C^2_b(H)$ with
\begin{enumerate}[(i)]
\item First order derivative:
\begin{eqnarray}
\langle DR_tf(x), h\rangle=-\int_{H}f(z)\left(\sum^{\infty}_{k=1}\frac{p'_{\alpha}(\frac{z_k-e^{-\gamma_k t}x_k}{c_k(t)})}{p_{\alpha}(\frac{z_k-e^{-\gamma_k t}x_k}{c_k(t)})}\frac{e^{-\gamma_k t} h_k}{c_k(t)}\right)\mu^0_t(dz)\label{DR_tf}
\end{eqnarray}
where $\mu^0_t$ is the law of $Y^0_{t}=Z_{A}(t)$, and
\begin{align}
\sup_{x\in H}|\langle DR_t f(x), h\rangle|\leq c_\alpha \Lambda_t\|f\|_0, \quad  \text{where} \quad c_\alpha=\int_{\mathbb{R}}\frac{p'_{\alpha}(z)^2}{p_{\alpha}(z)}dz,\label{first}
\end{align}
and $\Lambda_t$ is given by (\ref{im10}).
\item Second order derivative:
\begin{eqnarray}
\langle D^2 R_tf(x)h, g\rangle &=& -\int_{H}f(z+e^{tA}x)\sum^{\infty}_{l=1}\left(\sum_{k\neq l}\frac{p'_{\alpha}(\frac{z_k}{c_k(t)})p'_{\alpha}(\frac{z_l}{c_l(t)})}{p_{\alpha}(\frac{z_k}{c_k(t)})p_{\alpha}(\frac{z_l}{c_l(t)})}\frac{e^{-\gamma_k t}e^{-\gamma_l t} h_k g_l}{c_k(t)c_l(t)}\right.\nonumber\\
&&+\left.\frac{p''_{\alpha}(\frac{z_l}{c_l(t)})}{p_{\alpha}(\frac{z_l}{c_l(t)})}\frac{e^{-2\gamma_l t} h_l g_l}{c^2_l(t)}\right)\mu^0_t(dz),\label{D^2f}
\end{eqnarray}
and
\begin{align}
\sup_{x\in H}|\langle D^2R_t f(x)h, g\rangle|\leq \tilde{c}_\alpha \Lambda^2_t\|f\|_0, \label{second}
\end{align}
where $\tilde{c}_\alpha=\sqrt{2}\max\left\{\int_{\mathbb{R}}\frac{p'_{\alpha}(z)^2}{p_{\alpha}(z)}dz, \left[\int_{\mathbb{R}}\frac{p''_{\alpha}(z)^2}{p_{\alpha}(z)}dz\right]^{1/2}\right\}$.
\item H\"older continuity: for any $\beta>0$ and $r\in(0,1)$,
\begin{align}
|\<DR_t f(x)-DR_t f(y),h\>|\leq C_\alpha {\Lambda_t}^{1+r-\beta}\|f\|_\beta|x-y|^r, \quad\forall x,y\in H,  \label{es5}
\end{align}
where $C_\alpha>0$ is a constant.
\end{enumerate}
\end{theorem}

\begin{proof}
The results that $R_t f\in C^1_b(H)$, (\ref{DR_tf}), (\ref{first}) have been proved in \cite[Theorem 4.14]{EJ1}, it suffices to prove (\ref{D^2f})-(\ref{es5}). In order to show (\ref{D^2f}) and (\ref{second}), we mainly follow the steps in \cite[Theorem 4.14]{EJ1}. So, we only consider the case that $f\in C_b(H)$ is cylindrical, i.e.,
$$
f(x)=\tilde{f}(x_1,\ldots,x_j),\quad x\in H
$$
for some $j\geq 1$ and $\tilde f:\mR^j\rightarrow\mR$. We also assume that $\tilde{f}$ has bounded support in $\RR^j$. Then the general case of $f$ can be proved by an argument of approximation (see \text{Step II-Step V} in \cite[Theorem 4.14]{EJ1}).

Fix arbitrary $x, h, g\in H$ with $|h|\leq1, |g|\leq 1$. we will show that there exists $D^2_{h\otimes g}R_t f(x)$, the directional derivative of $DR_t f(x)h$, along the direction $g$ at $x$. To shorten the notation, we write
$$
\xi_{k,l}(z):=\frac{p'_{\alpha}(\frac{z_k-e^{-\gamma_k t}x_k}{c_k(t)})p'_{\alpha}(\frac{z_l-e^{-\gamma_l t}x_l}{c_l(t)})}{p_{\alpha}(\frac{z_k-e^{-\gamma_k t}x_k}{c_k(t)})p_{\alpha}(\frac{z_l-e^{-\gamma_l t}x_l}{c_l(t)})}.
$$
Let $g^N=\sum^N_{k=1}g_k e_k$, for any $m\geq \max\{j, N\}$. Then by (\ref{DR_tf}), we get
\begin{eqnarray*}
\langle D^2 R_tf(x)h, g^N\rangle &=& -\int_{\RR^m}\tilde{f}(z)\sum^{N}_{l=1}\left(\sum_{k\neq l}\xi_{k,l}(z)\frac{e^{-\gamma_k t}e^{-\gamma_l t} h_k g_l}{c_k(t)c_l(t)}\right.\\
&&\qquad+\left.\frac{p''_{\alpha}(\frac{z_l-e^{-\gamma_l t}x_l}{c_l(t)})}{p_{\alpha}(\frac{z_l-e^{-\gamma_l t}x_l}{c_l(t)})}\frac{(e^{-\gamma_l t})^2 h_l g_l}{c^2_l(t)}\right)\prod^m_{l=1}\mu^{x_l}_t(dz_l)\\
&=&-\int_{H}f(z)\sum^{N}_{l=1}\left(\sum_{k\neq l}\xi_{k,l}(z)\frac{e^{-\gamma_k t}e^{-\gamma_l t} h_k g_l}{c_k(t)c_l(t)}\right.\\
&&\qquad+\left.\frac{p''_{\alpha}(\frac{z_l-e^{-\gamma_l t}x_l}{c_l(t)})}{p_{\alpha}(\frac{z_l-e^{-\gamma_l t}x_l}{c_l(t)})}\frac{e^{-2\gamma_l t} h_l g_l}{c^2_l(t)}\right)\mu^x_t(dz).\\
\end{eqnarray*}
In order to pass to the limit, as $N\rightarrow \infty$, we show that
\begin{eqnarray}
\phi_N (t,x)&:=&\sum^{N}_{l=1}\left(\sum_{k\neq l}\xi_{k,l}(z)\frac{e^{-\gamma_k t}e^{-\gamma_l t} h_k g_l }{c_k(t)c_l(t)}\right.\nonumber\\
&&\qquad\left.+\frac{p''_{\alpha}(\frac{z_l-e^{-\gamma_l t}x_l}{c_l(t)})}{p_{\alpha}(\frac{z_l-e^{-\gamma_l t}x_l}{c_l(t)})}\frac{e^{-2\gamma_l t} h_l g_l }{c^2_l(t)}\right)\quad \text{converges in} \quad L^2(\mu^x_t).\label{L^2 convergence}
\end{eqnarray}
In fact, notice that for any $k\neq l$,
\begin{eqnarray*}
&&\int_{H}\xi_{k,l}(z)\cdot\frac{p''_{\alpha}(\frac{z_l-e^{-\gamma_l t}x_l}{c_l(t)})}{p_{\alpha}(\frac{z_l-e^{-\gamma_l t}x_l}{c_l(t)})}\mu^x_t(dz)\\
&=& \int_{\RR}\frac{p'_{\alpha}(\frac{z_k-e^{-\gamma_k t}x_k}{c_k(t)})}{p_{\alpha}(\frac{z_k-e^{-\gamma_k t}x_k}{c_k(t)})}\mu^{x_k}_t(dz_k)\int_{\RR}\frac{p'_{\alpha}(\frac{z_l-e^{-\gamma_l t}x_l}{c_l(t)})}{p_{\alpha}(\frac{z_l-e^{-\gamma_l t}x_l}{c_l(t)})}\cdot\frac{p''_{\alpha}(\frac{z_l-e^{-\gamma_l t}x_l}{c_l(t)})}{p_{\alpha}(\frac{z_l-e^{-\gamma_l t}x_l}{c_l(t)})}\mu^{x_l}_t(dz_l)\\
&=& \int_{\RR}p'_{\alpha}(y)dy\cdot\int_{\RR}\frac{p'_{\alpha}(\frac{z_l-e^{-\gamma_l t}x_l}{c_l(t)})}{p_{\alpha}(\frac{z_l-e^{-\gamma_l t}x_l}{c_l(t)})}\cdot\frac{p''_{\alpha}(\frac{z_l-e^{-\gamma_l t}x_l}{c_l(t)})}{p_{\alpha}(\frac{z_l-e^{-\gamma_l t}x_l}{c_l(t)})}\mu^{x_l}_t(dz_l)\\
&=& 0,
\end{eqnarray*}
where the last inequality by the fact that $p'_{\alpha}$ is odd. Then, for any $N,p\in \NN$,
\begin{eqnarray*}
&&\int_{H}\!\!\left|\sum^{N+p}_{l=N}\!\!\left(\sum_{k\neq l}\xi_{k,l}(z)\frac{e^{-\gamma_k t}e^{-\gamma_l t} h_k g_l }{c_k(t)c_l(t)}+\frac{p''_{\alpha}(\frac{z_l-e^{-\gamma_l t}x_l}{c_l(t)})}{p_{\alpha}(\frac{z_l-e^{-\gamma_l t}x_l}{c_l(t)})}\frac{(e^{-\gamma_l t})^2 h_l g_l }{c^2_l(t)}\right)\!\!\right|^2\mu^x_t(dz)\\
&=&\int_{H}\sum^{N+p}_{l=N}\sum_{k\neq l}\left(\xi_{k,l}(z)\frac{e^{-\gamma_k t}e^{-\gamma_l t} h_k g_l }{c_k(t)c_l(t)}\right)^2+\sum^{N+p}_{l=N}\left(\frac{p''_{\alpha}(\frac{z_l-e^{-\gamma_l t}x_l}{c_l(t)})}{p_{\alpha}(\frac{z_l-e^{-\gamma_l t}x_l}{c_l(t)})}\frac{(e^{-\gamma_l t})^2 h_l g_l }{c^2_l(t)}\right)^2\mu^x_t(dz)\\
&=&\sum^{N+p}_{l=N}\sum_{k\neq l}\left(\frac{e^{-\gamma_k t}e^{-\gamma_l t}}{c_k(t)c_l(t)}\right)^2 h^2_k g^2_l \int_{\RR}\frac{p'_{\alpha}(y_l)^2}{p_{\alpha}(y_l)}dy_l\int_{\RR}\frac{p'_{\alpha}(y_k)^2}{p_{\alpha}(y_k)}dy_k+\sum^{N+p}_{l=N}\left(\frac{e^{-\gamma_l t}}{c^2_l(t)}\right)^4 h^2_l g^2_l \int_{\RR}\frac{p''_{\alpha}(y_l)}{p_{\alpha}(y_l)}dy_l\\
&\leq& \tilde{c}^2_{\alpha}\Lambda^4_t|h|^2\sum^{N+p}_{l=N}g^2_l,
\end{eqnarray*}
where $\tilde{c}_\alpha=\sqrt{2}\max\left\{\int_{\mathbb{R}}\frac{p'_{\alpha}(z)^2}{p_{\alpha}(z)}dz, \left[\int_{\mathbb{R}}\frac{p''_{\alpha}(z)^2}{p_{\alpha}(z)}dz\right]^{1/2}\right\}$.

Note that, for any $N\in\NN$,
\begin{eqnarray*}
\langle D^2 R_tf(x)h, g^N\rangle&=&-\int_{H}f(z+e^{tA}x)\sum^{N}_{l=1}\left(\sum_{k\neq l}\frac{p'_{\alpha}(\frac{z_k}{c_k(t)})p'_{\alpha}(\frac{z_l}{c_l(t)})}{p_{\alpha}(\frac{z_k}{c_k(t)})p_{\alpha}(\frac{z_l}{c_l(t)})}\frac{e^{-\gamma_k t}e^{-\gamma_l t} h_k g_l}{c_k(t)c_l(t)}\right.\\
&&+\left.\frac{p''_{\alpha}(\frac{z_l}{c_l(t)})}{p_{\alpha}(\frac{z_l}{c_l(t)})}\frac{(e^{-\gamma_l t})^2 h_l g_l}{c^2_l(t)}\right)\mu^0_t(dz).
\end{eqnarray*}
Up to now we can showed that
\begin{eqnarray}
\frac{DR_t f(x+\epsilon g^N)h-DR_t f(x)h}{\epsilon}=\frac{1}{\epsilon}\int^{\epsilon}_0 \langle D^2R_t f(x+r g^N)h, g^N\rangle dr\label{3.4}
\end{eqnarray}
Using (\ref{L^2 convergence}), it is easy to see that, for any $r\in (0,1)$, $N\in\NN$,
\begin{eqnarray}
&&\lim_{N\rightarrow \infty}\langle D^2 R_tf(x+rg^N)h, g^N\rangle\nonumber\\
&=&-\int_{H}f(z+e^{tA}(x+rg))\sum^{\infty}_{l=1}\left(\sum_{k\neq l}\frac{p'_{\alpha}(\frac{z_k}{c_k(t)})p'_{\alpha}(\frac{z_l}{c_l(t)})}{p_{\alpha}(\frac{z_k}{c_k(t)})p_{\alpha}(\frac{z_l}{c_l(t)})}\frac{e^{-\gamma_k t}e^{-\gamma_l t} h_k g_l}{c_k(t)c_l(t)}\right.\nonumber\\
&&+\left.\frac{p''_{\alpha}(\frac{z_l}{c_l(t)})}{p_{\alpha}(\frac{z_l}{c_l(t)})}\frac{(e^{-\gamma_l t})^2 h_l g_l}{c^2_l(t)}\right)\mu^0_t(dz).\label{3.5}
\end{eqnarray}
Moreover, for any $r\in (0,1)$, $|\langle D^2 R_tf(x+rg^N)h, g^N\rangle|\leq\tilde{c}_\alpha \Lambda^2_t\|f\|_0$. Then by dominated convergence theorem in (\ref{3.4}), we obtain
\begin{eqnarray*}
\frac{DR_t f(x+\epsilon g)h-DR_t f(x)h}{\epsilon}=\frac{1}{\epsilon}\int^{\epsilon}_0 u(t,h,x+rg)dr,\quad \epsilon\in(0,1),
\end{eqnarray*}
where $u(t,x+rg)$ is the right-hand side of (\ref{3.5}). This shows that $DR_tf(x)h$ is G\^{a}teaux differentiable at $x\in H$ along the direction $h$ and (\ref{D^2f}), (\ref{second}) hold.

\vspace{2mm}
We proceed to show (\ref{es5}). It is easy to see that
\begin{align}
\sup_{x\in H}|\< DR_t f(x), h\>|\leq \|f\|_1.   \label{es1}
\end{align}
By taking $\nabla f$ into (\ref{first}), we can also obtain
\begin{align}
\sup_{x\in H}|\< D^2R_t f(x)h, g\>|\leq c_\alpha \Lambda_t\|f\|_1.    \label{es2}
\end{align}
Thus, we have by (\ref{first}) and (\ref{es1}) that
\begin{align*}
\sup_{x\in H}|\< DR_t f(x), h\>|\leq \hat c_\alpha \Lambda_t^{1-\beta}\|f\|_\beta,
\end{align*}
and similarly, it holds by (\ref{second}) and (\ref{es2}) that
\begin{align*}
\sup_{x\in H}|\< D^2R_t f(x)h, g\>|\leq \hat c_\alpha {\Lambda_t}^{2-\beta}\|f\|_\beta.
\end{align*}
Combing the above two inequalities with the interpolation theory, we can get the desired result.
\end{proof}

\begin{remark}
By the similar argument above, we could obtain $R_t f\in C^{\infty}_b(H)$, for any $t>0$, $f\in B_b(H)$. Moreover, there exists a positive constant $C_n$ such that
$$
\sup_{x\in H}\|D^n R_t f(x)\|\leq C_n \Lambda^n_t\|f\|_0.
$$
\end{remark}

\begin{remark}\label{example}
If there exists $r\geq -1/\alpha$ such that
$$
\beta_n\geq C \gamma_n^{-r},\quad \forall n\geq 1,
$$
then we can give a upper bound of $\Lambda_t$ in Theorem \ref{T3.1}, i.e.,
$$
\Lambda_t=\sup_{n\geq1}\frac{e^{-{\gamma_n t}}\gamma^{1/\alpha}_n}{\beta_n}\leq \frac{C}{t^{r+\frac{1}{\alpha}}},\quad t>0.
$$
\end{remark}

\subsection{\bf Elliptic Kolmogorov equation}

In this subsection, we always assume that (i), (ii)' and (iv) hold.
%\begin{description}
%  \item[(v)'] For any $r<\alpha$ and $\lambda>0$, we have
%\begin{align}
%\int_0^\infty\e^{-\lambda t}\Lambda_t^r \dif t<\infty.    \label{im}
%\end{align}
%\end{description}
Now, we study the following Kolmogorov equation in $H$
\begin{eqnarray}
\lambda U-\langle B, DU\rangle- \sL U=F,\label{4.1}
\end{eqnarray}
where $\lambda>0$, $F\in B_b(H,H)$ and the operator $\sL$ is the infinitesimal generator of OU-semigroup $R_t$, i.e.,
 $$\sL U(x)=\langle Ax, DU(x)\rangle+\sum_{k=1}\beta^{\alpha}_k\int_{\RR}[U(x+e_k z)-U(x)-\langle DU(x), e_k z\rangle1_{\{|z|\leq 1\}}]\frac{c_\alpha}{|z|^{1+\alpha}}dz.$$

\begin{theorem}\label{pde}
Assume that $B\in C_b^\beta(H, H)$ for some $\beta>0$. Then, for $\lambda$ big enough, every $F\in C_b^\beta(H, H)$ and any $0<\theta<\beta$, there exists a function $U\in C_b^{\gamma+\theta}(H,H)$ satisfying the following integral equation:
\begin{align}
U(x)=\int_0^\infty\!\e^{-\lambda t}R_t\Big(\<B, DU\>+F\Big)(x)\dif t.   \label{inte}
\end{align}
Moreover, $U$ also solves equation (\ref{4.1}) and we have
\begin{align}
\|U\|_{\gamma+\theta}\leq C_{\lambda} \|F\|_{\beta},  \label{est}
\end{align}
where $C_\lambda$ is a positive constant satisfying $\lim_{\lambda\rightarrow +\infty}C_\lambda=0$.
\end{theorem}

\begin{proof}
Let us first show the estimate (\ref{est}). In fact, let $U$ satisfies (\ref{inte}), without loss of generality, we may assume that $1+\beta<\gamma+\theta<2$. Then, by (\ref{es5}) and condition (v) we have
\begin{align*}
\|U\|_{\gamma+\theta}&\leq \int_0^\infty\!\e^{-\lambda t}\Big\|R_t\Big(\<B, DU\>+F\Big)\Big\|_{\gamma+\theta}\dif t\\
&\leq\int_0^\infty\!\e^{-\lambda t}\Lambda_t^{\gamma+\theta-\beta}\dif t\cdot\|\<B, DU\>+F\|_{\beta}\\
&\leq C_{\lambda}\Big(\|F\|_{\beta}+\|B\|_{\beta}\cdot\|U\|_{1+\beta}\Big),
\end{align*}
where $C_\lambda$ is given by
$$
C_\lambda:=\int_0^\infty\!\e^{-\lambda t}\Lambda_t^{\gamma+\theta-\beta}\dif t,
$$
and by dominate convergence theorem it holds that $\lim_{\lambda\rightarrow +\infty}C_\lambda=0$. Take $\lambda$ big enough such that
$$
C_\lambda\|B\|_{\beta}<\frac{1}{2},
$$
we get the desired result.

Now, we construct the solution of (\ref{inte}) via Picard's iteration argument.
Set $U_0\equiv0$ and for $n\in \mN$, define $U_n$ recursively by
$$
U_n(x):=\int_0^\infty\!\e^{-\lambda t}R_t\Big(\<B,DU_{n-1}\>+F\Big)(x)\dif t.
$$
In view of (\ref{first}), it is easy to check that $U_1\in C^1_b(H, H)$, and $U_2$ is thus well defined, and so on. We show that $U_1\in C_b^{\gamma+\theta}(H, H)$ with any $\theta\in(0,\beta)$. In fact, thanks to (v) and using (\ref{es5}) once again, we can deduce that
\begin{align*}
DU_1(x)-DU_1(y)&=\int_0^\infty\!\e^{-\lambda t}\Big[DR_tF(x)-DR_tF(y)\Big]\dif t\\
&\leq c_{\alpha}|x-y|^{\gamma+\theta-1}\int_0^\infty\!\e^{-\lambda t}\Lambda_t^{\gamma+\theta-\beta}\dif t\cdot\|F\|_\beta=c_\alpha C_{\lambda}|x-y|^{\gamma+\theta-1}\|F\|_\beta.
\end{align*}
Notice that $\gamma+\theta-1>\beta$. As a result,
$$
\<B, DU_1\>\in C_b^\beta(H,H).
$$
Repeating the above argument, we have for every $n\in\mN$ and any $\theta\in(0,\beta)$,
$$
U_n\in C_b^{\gamma+\theta}(H, H).
$$
Moreover, for any $n>m$
\begin{align*}
U_n(x)-U_m(x)&=\int_0^\infty\!\e^{-\lambda t}R_t\Big(\langle B, DU_{n-1}-DU_{m-1}\rangle\Big)(x)\dif t,
\end{align*}
we further have that
\begin{align*}
\|U_n-U_m\|_{\gamma+\theta}&\leq \int_0^\infty\!\e^{-\lambda t}\Big\|R_t\Big(\langle B, DU_{n-1}-DU_{m-1}\rangle\Big)\Big\|_{\gamma+\theta}\dif t\\
&\leq c_{\alpha}\!\!\int_0^\infty\!\e^{-\lambda t}\Lambda^{\gamma+\theta-\beta}_t\dif t\cdot\|\langle B, DU_{n-1}-DU_{m-1}\rangle\|_{\beta}\\
&\leq 2c_{\alpha}C_{\lambda}\|B\|_{\beta}\cdot\|DU_{n-1}-DU_{m-1}\|_{\beta}\\
&\leq 2c_{\alpha}C_{\lambda}\|B\|_{\beta}\cdot\|U_{n-1}-U_{m-1}\|_{\gamma+\theta},
\end{align*}
This means that for $\lambda$ big enough, $U_n$ is Cauchy sequence in $C_b^{\gamma+\theta}(H, H)$. Thus, there exists a limit function $U\in C_b^{\gamma+\theta}(H, H)$ with $\theta\in(0,\beta)$ satisfying (\ref{inte}). The assertion that $U$ solves (\ref{4.1}) follows by integral by part. The whole proof is finished.
\end{proof}

\br
It may be expected that the optimal regularity of $U$ should be $C_b^{\alpha+\beta}(H,H)$. However, due to the uncertain of $\Lambda_t$, we can not obtain thus estimate. Nevertheless, this is enough for us to prove the pathwise uniqueness for SPDE (\ref{main equation}).
\er

\section{\bf Strong uniqueness}

We call a predictable $H$-valued stochastic process $\{X_t\}_{t\geq 0}$ depending on initial value $x\in H$, is a mild solution of equation (\ref{main equation}) on the filtered probability space $(\Omega,\mathcal{F}_t,\mathbb{P})$, if $\{X_t\}_{t\geq 0}$ is an $\mathcal{F}_t$-adapted and satisfies
\begin{eqnarray}
X_t=e^{tA}x+\int^t_0 e^{(t-s)A}B(X_s)ds+\int^t_0 e^{(t-s)A}dZ_s,\label{mild solution 1}
\end{eqnarray}
where the deterministic integral in (\ref{mild solution 1}) is well defined by the assumption that $B$ is bounded.
The existence of a solution to equation (\ref{main equation}) is known under our assumptions. Thus, by the classical Yamada-Watanabe principle \cite{Ik-Wa}, we only need to focus on the pathwise uniqueness.

\vskip 0.3cm
Usually, the It\^o's formula is performed for functions $f\in C_b^2(H, H)$. However, this is too strong for our latter use. Actually, $Z_t$ is a $\alpha$-stable process in our case, and we will show that It\^o's formula holds for $f\in C_b^{r}(H, H)$ with any $r>\alpha$, which is stated as follows:
%Firstly, let's give some nations, which will be used in the It\^o's formula.
%\vskip 0.3cm
%For $t>0$ and $\Gamma\in\mathfrak{B}(\mathbb{R}\setminus\{0\})$, the Poisson random measure associated with $Z^{n}_t$ is defined by
%$$N^{(n)}(t,\Gamma)=\sum_{s\leq t}I_{\Gamma}(Z^{n}_s-Z^{n}_{s-}),$$
%which the compensated Poisson measure is given by
%$$\widetilde{N}^{(n)}(t,\Gamma)=N^{(n)}(t,\Gamma)-t\nu(\Gamma).$$
%By L\'evy-It\^o's decomposition, one has
%$$Z^{n}_{t}=\int_{|x|\leq1}x\widetilde{N}^{(n)}(t,dx)+\int_{|x|>1}x N^{(n)}(t,dx).$$

%\vskip 0.3cm
%Then the It\^o's formula is stated as follows:
\bl\label{ito}
Let $X_t$ satisfies equation (\ref{main equation}) and $f\in C^{r}_b(H, H)$ with $r>\alpha$. Then, we have
\begin{eqnarray*}
f(X_t)&=&f(x)+\int^t_0\langle AX_s+B(X_s), Df(X_s)\rangle ds\\
&&+\sum_{k=1}\beta^{\alpha}_k\int^t_0\int_{\RR}[f(X_s+e_k z)-f(X_s)-\langle Df(X_s), e_k z\rangle1_{\{|z|\leq 1\}}]\frac{1}{|z|^{1+\alpha}}dzds\\
&&+\sum_{k=1}\int^t_0\int_{\RR}\left[f(X_{s-}+\beta_k e_k z)-f(X_{s-})\right]\tilde{N}^{(k)}(dz, ds).
\end{eqnarray*}
\el
\begin{proof}
Let $f_n\in C^{2}_b(H, H)$ with $\|f_n\|_{r}\leq \|f\|_{r}$ and $\|f_n-f\|_{r'}\rightarrow 0$ for every $r'<r$. Then we can use It\^o's formula for $f_n(X_t)$, i.e.,
\begin{align*}
f_n(X_t)&=f_n(x)+\int^t_0\langle AX_s+B(X_s), Df_n(X_s)\rangle ds\\
&\quad+\sum_{k=1}\beta^{\alpha}_k\int^t_0\int_{\RR}[f_n(X_s+e_k z)-f_n(X_s)-\langle Df_n(X_s), e_k z\rangle1_{\{|z|\leq 1\}}]\frac{1}{|z|^{1+\alpha}}dzds\\
&\quad+\sum_{k=1}\int^t_0\int_{\RR}\left[f_n(X_{s-}+\beta_k e_k z)-f_n(X_{s-})\right]\tilde{N}^{(k)}(dz, ds)\\
&=:\mI_1+\mI_2+\mI_3.
\end{align*}
Now we are going to pass the limits on the both sides of the above equality. For any $x\in H$, it is easy to see that, as $n\rightarrow \infty$,
$$
\mI_1\rightarrow f(x)+\int^t_0\langle AX_s+B(X_s), Df(X_s)\rangle ds.
$$
Thanks to the assumption that $\sum_{n\geq1}\beta^{\alpha}_{n}<\infty$ and in view of the following estimates:
$$
|f_n(x+e_kz)-f_n(x)-\langle Df_n(x), e_kz\rangle|\leq \|f_n\|_{r}|z|^{r}\leq \|f\|_{r}|z|^{r},\quad  |z|\leq 1
$$
and $$
|f_n(x+e_kz)-f_n(x)|\leq 2\|f_n\|_{0}\leq 2\|f\|_0,\quad |z|\geq 1,
$$
we obtain by dominated convergence theorem that, as $n\rightarrow\infty$,
$$
\mI_2\rightarrow\sum_{k=1}\beta^{\alpha}_k\int^t_0\int_{\RR}[f(X_s+e_k z)-f(X_s)-\langle Df(X_s), e_k z\rangle1_{\{|z|\leq 1\}}]\frac{1}{|z|^{1+\alpha}}dzds.
$$
Finally, by the isometry formula we have
\begin{align*}
&\mE\left|\sum_{k=1}\int^t_0\!\!\int_{\RR}\left[f_n(X_{s-}+\beta_k e_k z)-f_n(X_{s-})-f(X_{s-}+\beta_k e_k z)+f(X_{s-})\right]\tilde{N}^{(k)}(dz, ds)\right|^2\\
&=\sum_{k=1}\int^t_0\!\!\int_{\RR}\EE\left|f_n(X_{s}+\beta_k e_k z)-f_n(X_{s})-f(X_{s-}+\beta_k e_k z)+f(X_{s-})\right|^2\nu(dz)ds\\
&=\sum_{k=1}\beta^{\alpha}_k\int^t_0\!\!\int_{\RR}\EE\left|f_n(X_{s}+ e_k z)-f_n(X_{s})-f(X_{s}+ e_k z)+f(X_{s})\right|^2\frac{1}{|z|^{1+\alpha}}dzds\\
&\rightarrow 0,\quad n\rightarrow\infty.
\end{align*}
The proof is finished.
\end{proof}

\vskip 0.3cm
Now,  assume $B\in C^{\beta}_{b}(H,H)$ and let $U$ solves the following equation:
\begin{align*}
\lambda U-\<B,DU\>-\sL U=B.
\end{align*}
According to Theorem \ref{pde}, we have $U\in C_b^{\gamma+\theta}(H, H)$ with $\theta<\beta$.
We prove the following Zvonkin's transformation.

\bl
Let $X_t$ be a solution of equation \ref{main equation}, then we have \begin{eqnarray}
X_t&=&e^{tA}(x+U(x))+\int^t_0 e^{(t-s)A}\lambda U(X_s)ds-U(X_t)-\int^t_0 Ae^{(t-s)A}U(X_s)ds\nonumber\\
&&\!\!+\sum_{k=1}\int^t_0\int_{\RR}e^{(t-s)A}\left[U(X_{s-}+\beta_k e_k z)-U(X_{s-})\right]\tilde{N}^{(k)}(ds, dz)+\int^t_0 e^{(t-s)A}dZ_s.\label{mild}
\end{eqnarray}
\el

\begin{proof}
Thanks to Lemma \ref{ito}, we can use the It\^o's formula for $U(X_t)$ to get that
\begin{eqnarray*}
dU(X_t)\!\!&=&\!\!\langle B(X_t), DU(X_t)\rangle dt+\sL U(X_t)dt\\
&\quad+&\sum_{k=1}\int_{\RR}\left[U(X_{t-}+\beta_k e_k z)-U(X_{t-})\right]\tilde{N}^{(k)}(t, dz)\\
&=& \lambda U(X_t)dt-B(X_t)dt+\sum_{k=1}\int_{\RR}\left[U(X_{t-}+\beta_k e_k z)-U(X_{t-})\right]\tilde{N}^{(k)}(dt, dz),
\end{eqnarray*}
which give a formula for $B(X_t)dt$:
$$
B(X_t)dt=\lambda U(X_t)dt-dU(X_t)+\sum_{k=1}\int_{\RR}\left[U(X_{t-}+\beta_k e_k z)-U(X_{t-})\right]\tilde{N}^{(k)}(dt, dz).
$$
We put this formula in equation (\ref{main equation}) and get
\begin{eqnarray*}
dX_t&=&AX_tdt+\lambda U(X_t)dt-dU(X_t)\nonumber\\
&&+\sum_{k=1}\int_{\RR}\left[U(X_{t-}+\beta_k e_k z)-U(X_{t-})\right]\tilde{N}^{(k)}(dt, dz)+dZ_t,
\end{eqnarray*}
then follow the usual variation of constant method and get
\begin{eqnarray*}
X_t&=&e^{tA}x+\int^t_0 e^{(t-s)A}\lambda U(X_s)ds-\int^t_0 e^{(t-s)A}dU(X_s)\nonumber\\
&&+\sum_{k=1}\int^t_0\int_{\RR}e^{(t-s)A}\left[U(X_{s-}+\beta_k e_k z)-U(X_{s-})\right]\tilde{N}^{(k)}(ds, dz)+\int^t_0 e^{(t-s)A}dZ_s.
\end{eqnarray*}
Finally, the integration by parts formula implies
\begin{eqnarray*}
X_t&=&e^{tA}(x+U(x))+\int^t_0 e^{(t-s)A}\lambda U(X_s)ds-U(X_t)-\int^t_0 Ae^{(t-s)A}U(X_s)ds\\
&&\!\!+\sum_{k=1}\int^t_0\int_{\RR}e^{(t-s)A}\left[U(X_{s-}+\beta_k e_k z)-U(X_{s-})\right]\tilde{N}^{(k)}(ds, dz)+\int^t_0 e^{(t-s)A}dZ_s.
\end{eqnarray*}
The proof is complete.
\end{proof}

\vskip 0.3cm
The following result was proved in \cite{P} in finite dimensional. For the sake of completeness, we provide a simple proof here.
\bl
For every $x,y\in H$ and any $\theta<\beta$, we have
\begin{align}
|U(x+z)-U(x)-U(y+z)+U(y)|\leq |x-y|\cdot|z|^{\gamma+\theta-1}\|U\|_{\gamma+\theta}.  \label{unun}
\end{align}
\el

\begin{proof}
Set
$$
J_zU(x):=U(x+z)-U(x).
$$
It is obvious that for any $|z|\leq1$,
$$
\|D J_zU\|_0\leq |z|^{\gamma+\theta-1}\|DU\|_{\gamma+\theta-1}.
$$
Thus, we can deduce that
$$
|J_zU(x)-J_zU(y)|\leq |x-y|\cdot\|DJ_zU\|_0\leq |x-y|\cdot|z|^{\gamma+\theta-1}\|U\|_{\gamma+\theta}.
$$
The proof is complete.
\end{proof}

\vspace{0.5cm}
We are now in the position to give the proof of our main result.\\
%\begin{theorem}
%Under the assumption (i)-(vi), pathwise uniqueness hold for equation (\ref{main equation}).
%\end{theorem}
\textbf{Proof of Theorem \ref{main1}:}
Since uniqueness is a local property, we only need to consider on the interval $[0,T]$ with $T$ small. Let $X_t$ and $Y_t$ be two solutions of equation (\ref{main equation}) both starting from $x\in H$. Then, by (\ref{mild}) we have $V_t:=X_t-Y_t$ satisfies the following equation:
\begin{align*}
V_t&=\lambda\int^t_0 e^{(t-s)A}\Big(U(X_s)-U(Y_s)\Big)ds-\big[U(X_t)-U(Y_t)\big]\\
&\quad-\int^t_0 Ae^{(t-s)A}\Big(U(X_s)-U(Y_s)\Big)ds+I_t(X)-I_t(Y),
\end{align*}
where
$$I_t(X)=\sum_{k=1}\int^t_0\!\!\int_{\RR}e^{(t-s)A}\Big[U(X_{s-}+\beta_k e_k z)-U(X_{s-})\Big]\tilde{N}^{(k)}(ds, dz)$$
and
$$I_t(Y)=\sum_{k=1}\int^t_0\!\!\int_{\RR}e^{(t-s)A}\Big[U(Y_{s-}+\beta_k e_k z)-U(Y_{s-})\Big]\tilde{N}^{(k)}(ds, dz).$$
In view of (\ref{est}), we have
$$
\|U\|_{\gamma+\theta}\leq C_\lambda\|B\|_\beta.
$$
Notice that by the maximal inequality,
$$
\left\|\int^{\cdot}_0 Ae^{(\cdot-s)A}f_sds\right\|^2_{L^2(0,T; H)}\leq C_T\|f\|^2_{L^2(0,T; H)},
$$
where $C_T$ is a constant independent of $f$. Then we have the following estimate:
\begin{align*}
\int^T_0 |V_t|^2dt&\leq4\int^T_0 \left|\lambda\int^t_0 e^{(t-s)A}\Big(U(X_s)-U(Y_s)\Big)ds\right|^2dt\\
&+4\int^T_0 |U(X_t)-U(Y_t)|^2dt+4\int^T_0 \left|\int^t_0 Ae^{(t-s)A}\Big(U(X_s)-U(Y_s)\Big)ds\right|^2dt\\
&+4\int^T_0 |I_t(X)-I_t(Y)|^2dt\\
&\leq \left(4\lambda^2TC^2_{\lambda}\|B\|^2_{\beta}+4C^2_{\lambda}\|B\|^2_{\beta}+4C_T C^2_{\lambda}\|B\|^2_{\beta}\right)\int^T_0|V_t|^2dt\\
&+4\int^T_0 |I_t(X)-I_t(Y)|^2dt.
\end{align*}
Since $\lim_{\lambda\rightarrow\infty}C_\lambda=0$, taking $\lambda$ sufficient large such that $4C^2_{\lambda}\|B\|^2_{\beta}+4C_T C^2_{\lambda}\|B\|^2_{\beta}\leq1/3$ and $T$ small enough such that $4\lambda^2TC^2_{\lambda}\|B\|^2_{\beta}\leq 1/3$, then we have
\begin{eqnarray*}
\int^T_0 |V_t|^2dt \leq 12\int^T_0 |I_t(X)-I_t(Y)|^2dt.
\end{eqnarray*}
Meanwhile,
\begin{align*}
&\EE |I_t(X)-I_t(Y)|^2\\
&=\sum_{k=1}\int^t_0\!\!\int_{\RR}\EE\left|e^{(t-s)A}\Big[U(X_{s-}+\beta_k e_k z)-U(X_{s})-U(Y_{s}+\beta_k e_k z)+U(Y_{s-})\Big]\right|^2\nu(dz)ds\nonumber\\
&\leq \sum_{k=1}\beta^{\alpha}_k\int^t_0\sum_{m=1}e^{-2(t-s)\gamma_m}\int_{|z|\leq 1}\EE|X_s-Y_s|^2\cdot|z|^{2(\gamma+\theta-1)}\|U\|^2_{\gamma+\theta} \nu(dz)ds\\
&+\sum_{k=1}\beta^{\alpha}_k\int^t_0\sum_{m=1}e^{-2(t-s)\gamma_m}\int_{|z|> 1}2\|DU\|^2_{0}\EE|X_s-Y_s|^2 \nu(dz)ds
\end{align*}
Since we assumed that $\beta\in(1+\alpha/2-\gamma,1)$, there always exists a $\theta<\beta$ such that
$$
2(\gamma+\theta-1)>\alpha.
$$
As a result, we have by (\ref{unun})
\begin{eqnarray*}
\EE |I_t(X)-I_t(Y)|^2
&\leq& C_{\lambda}\|B\|^2_{\beta}\sum_{k=1}\beta^{\alpha}_k \int^t_0\sum_{m=1}e^{-2(t-s)\gamma_m}\EE|V_s|^2ds.
\end{eqnarray*}
Hence,
\begin{eqnarray*}
\EE \int^T_0|I_t(X)-I_t(Y)|^2dt&\leq& C_{\lambda}\|B\|^2_{\beta}\sum_{k=1}\beta^{\alpha}_k \int^T_0\int^t_0\sum_{m=1}e^{-2(t-s)\gamma_m}\EE|V_s|^2dsdt\\
&\leq& C_{\lambda}\|B\|^2_{\beta}\sum_{k=1}\beta^{\alpha}_k \int^T_0\left(\int^T_s\sum_{m=1}e^{-2(t-s)\gamma_m}dt\right)\EE|V_s|^2ds\\
&\leq& C_{\lambda}\|B\|^2_{\beta}\sum_{k=1}\beta^{\alpha}_k \left(\int^T_0\sum_{m=1}e^{-2t\gamma_m}dt\right)\int^T_0\EE|V_s|^2ds\\
&\leq& C_{\lambda}\|B\|^2_{\beta}\sum_{k=1}\beta^{\alpha}_k \sum_{m=1}\frac{1-e^{-2T\gamma_m}}{2\gamma_m}\int^T_0\EE|V_s|^2ds
\end{eqnarray*}
By assumption, $\sum_{m=1}\frac{1-e^{-2T\gamma_m}}{2\gamma_m}$ is finite and when $T\rightarrow0$, it converges to zero. Therefor we can obtain $\EE \int^T_0|I_t(X)-I_t(Y)|^2dt=0$ as $T$ small enough. This implies $X=Y$. The proof is complete.

\end{document}